\documentclass[letterpaper,12pt]{article}
\usepackage{amsmath,amssymb,amsfonts,epsfig}

\newtheorem{theo}{Theorem}[section]

\newtheorem{lemma}[theo]{Lemma}
\newtheorem{defn}[theo]{Definition}

\newtheorem{cor}[theo]{Corollary}

\newtheorem{counter}[theo]{Counterexample}

\newtheorem{rem}[theo]{Remarks}

\newenvironment{proof}{\noindent {\sc Proof}.}
                {\phantom{a} \hfill \framebox[2.2mm]{ } \bigskip}

\newcommand{\ZZ}{\mathbb{Z}}

\def\int{{\rm int}}

\newcommand{\G}{{\cal G}}
\renewcommand{\H}{{\cal H}}
\renewcommand{\L}{{\cal L}}

\renewcommand{\b}{\backslash}

\renewcommand{\phi}{\varphi}
\renewcommand{\theta}{\vartheta}

\newcommand\mset[1]{\left\{\!\!\left\{#1\right\}\!\!\right\}}  

\title{Hypergraphs: connection and separation}
\author{M. Amin Bahmanian\footnote{Email: mbahman@ilstu.edu. Phone: +309-438-8781  	ext. 7595. Mailing address: Department of Mathematics, Illinois State University, Stevenson Hall 313, Campus Box 4520, Normal, Illinois, 61790-4520, USA.} $\,$
and Mateja \v{S}ajna\footnote{Email: msajna@uottawa.ca. Phone: +613-562-5800 ext. 3522. Mailing address: Department of Mathematics and Statistics, University of Ottawa, 585 King Edward Avenue, Ottawa, ON, K1N 6N5,Canada.} \\ University of Ottawa}

\begin{document}
\maketitle \baselineskip 17pt

\begin{abstract}
In this paper we study fundamental connectivity properties of hypergraphs from a graph-theoretic perspective, with the emphasis on cut edges, cut vertices, and blocks. To prepare the ground, we define various types of subhypergraphs, as well as various types of walks in a hypergraph. We then prove a number of new results involving cut edges, cut vertices, and blocks. In particular, we describe the exact relationship between the block decomposition of a hypergraph and the block decomposition of its incidence graph.

\end{abstract}


\section{Introduction}

A data base search under ``hypergraph'' returns hundreds of journal articles published in the last couple of years alone, but only a handful of monographs. Among the latter, most either treat very specific problems in hypergraph theory (for example, colouring in \cite{Vol0} and even \cite{Vol}), or else are written with a non-mathematician audience in mind, and hence focus on applications (for example, \cite{Bre}). A mathematician or mathematics student looking for a general introduction to hypergraphs is left with Berge's decades-old {\em Hypergraphs} \cite{Ber1} and {\em Graphs and Hypergraphs} \cite{Ber}, and Voloshin's much more recent {\em Introduction to Graphs and Hypergraphs} \cite{Vol}, aimed at undergraduate students. The best survey on hypergraphs that we could find, albeit already quite out of date, is Duchet's chapter \cite{Duc} in the {\em Handbook on Combinatorics}. In particular, it describes the distinct paths that lead to the study of hypergraphs from graph theory, optimization theory, and extremal combinatorics, explaining the fragmented terminology and disjointed nature of the results. Berge's work, for example, though an impressive collection of results, shows a distinct bias for hypergraphs arising from extremal set theory and optimization theory, and as such is rather unappealing to a graph theorist.

The numerous journal publications, on the other hand, treat a great variety of specific problems on hypergraphs. Graph theorists find various ways of generalizing concepts from graph theory, often without justifying their own approach or comparing it with others. The same term in hypergraphs (for example, cycle)  may have a variety of different meanings. Sometimes, authors implicitly assume that results for graphs extend to hypergraphs. A coherent theory of hypergraphs, as we know it for graphs, is sorely lacking.

This paper can serve as an introduction to hypergraphs from a graph-theoretic perspective, with a focus on basic connectivity. To prepare the ground for the more involved results on block decomposition of hypergraphs, we needed to carefully and systematically examine the fundamental connectivity properties of hypergraphs, attempting to extend basic results such as those found in the first two chapters of a graph theory textbook. We are strongly biased in our approach by the second author's graph-theoretic perspective, as well as in our admiration for Bondy and Murty's graph theory ``bible'' \cite{BonMur} and its earlier incarnation \cite{BonMur0}.

This paper is organized as follows. In Section 2 we present the basic concepts involving hypergraphs, as well as some immediate observations. Section 3 forms the bulk of the work: from graphs to hypergraphs, we generalize the concepts of various types of walks, connection, cut edges and cut vertices, and blocks, and prove a number of new results involving these concepts.

\section{Basic definitions}

This section should serve as a rather comprehensive introduction to basic hypergraphs concepts. The graph-theoretic terms used in this article are either analogous to the hypergraph terms defined here, or else are standard and can be found in \cite{BonMur}.

\subsection{Hypergraphs}

We shall begin with some basic definitions, followed by a few comments on alternative terms found in the literature.

\begin{defn}{\rm
A hypegraph $H$ is an ordered pair $(V,E)$, where  $V$ and $E$ are disjoint finite sets such that $V \ne \emptyset$, together with a function  $\psi: E \rightarrow 2^V$, called the {\em incidence function}. The elements of $V=V(H)$ are called {\em vertices}, and the elements of $E=E(H)$ are called {\em edges}. The number of vertices $|V|$ and number of edges $|E|$ are called the {\em order} and {\em size} of the hypergraph, respectively. Often we denote $n=|V|$ and $m=|E|$. A hypergraph with a single vertex is called {\em trivial}, and a hypergraph with no edges is called {\em empty}.

Two edges $e,e' \in E$ are said to be {\em parallel} if $\psi(e)=\psi(e')$, and the number of edges parallel to edge $e$ (including $e$) is called the {\em multiplicity} of $e$. A hypergraph $H$ is called {\em simple} if no edge has multiplicity greater than 1; that is, if $\psi$ is injective.
}
\end{defn}
As is customary for graphs, the incidence function may be omitted when no ambiguity can arise (in particular, when the hypergraph is simple, or when we do not need to distinguish between distinct parallel edges). An edge $e$ is then identified with the subset $\psi(e)$ of $V$, and for $v \in V$ and $e \in E$, we then more conveniently write $v \in e$ or $v \not\in e$ instead of $v \in \psi(e)$ or $v \not\in \psi(e)$, respectively. Moreover, $E$ is then treated as a multiset, and we use double braces to emphasize this fact when needed. Thus, for example, $\{ 1,2\}=\mset{1,2}$ but $\{ 1,1,2\}=\{1,2\} \ne \mset{1,1,2}$.

\begin{defn}{\rm
Let $H=(V,E)$ be a hypergraph. If $v,w \in V$ are distinct vertices and there exists $e \in E$ such that $v,w \in e$, then $v$ and  $w$ are said to be {\em adjacent} in $H$ (via edge $e$). Similarly, if $e,f \in E$ are distinct (but possibly parallel) edges and $v \in V$ is  such that $v \in e \cap f$, then $e$ and $f$ are said to be {\em adjacent} in $H$ (via vertex $v$).

Each ordered pair $(v,e)$ such that $v \in V$, $e \in E$, and $v \in e$ is called a {\em flag} of $H$; the (multi)set of flags is denoted by $F(H)$. If $(v,e)$ is a flag of $H$, then we say that vertex $v$ is {\em incident} with edge $e$.

The {\em degree} of a vertex $v \in V$ (denoted by $\deg_H(v)$ or simply $\deg(v)$ if no ambiguity can arise) is the number of edges $e \in E$ such that $v \in e$. A vertex of degree 0 is called {\em isolated}, and a vertex of degree 1 is called {\em pendant}. A hypergraph $H$ is {\em regular of degree $r$} (or {\em $r$-regular}) if every vertex of $H$ has degree $r$.

The maximum (minimum) cardinality $|e|$ of any edge $e \in E$ is called the {\em rank} ({\em corank}, respectively) of $H$. A hypergraph $H$ is {\em uniform of rank} $r$ (or {\em $r$-uniform}) if $|e|=r$ for all $e \in E$. An edge $e \in E$ is called a {\em singleton edge} if $|e|=1$, and {\em empty} if $|e|=0$.
}
\end{defn}

\begin{rem}{\rm
In \cite{Vol,Bre}, a hypergraph is called {\em simple} if no edge is contained in another. In \cite{Duc,Bre}, an edge of cardinality 1 is called a {\em loop}. We shall not use this term, though, since in graph theory --- particularly in the context of connection --- it is more convenient to think of a loop as a multiset of vertices; that is, a loop contains a single vertex of multiplicity 2. Note that, while one could allow edges of a hypergraph to be multisets (rather than just sets) of vertices, we shall not consider this option.

Furthermore, in \cite{Bre}, two edges are called {\em incident} (rather than adjacent) if they share a vertex, and a vertex is adjacent to itself if it lies in a singleton edge. In addition, terms {\em empty hypergraph} and {\em trivial hypergraph} have a different meaning.
}
\end{rem}

The concepts of isomorphism and incidence matrix, to be defined below, are straightforward generalizations from graphs and designs.

\begin{defn}{\rm
Let $H_1=(V_1,E_1)$ and $H_2=(V_2,E_2)$ be hypergraphs with incidence functions $\psi_1$ and $\psi_2$, respectively. An {\em isomorphism} from $H_1$ to $H_2$ is a pair $(\phi,\theta)$ of bijections $\phi: V_1 \rightarrow V_2$ and $\theta: E_1 \rightarrow E_2$ such $\phi(\psi_1(e))=\psi_2(\theta(e))$ for all $e \in E_1$. Hypergraphs $H_1$ and $H_2$ are called {\em isomorphic} if there exists an isomorphism from $H_1$ to $H_2$.
}
\end{defn}
Omitting the incidence function, an isomorphism from $H_1$ to $H_2$ is simply a bijection $\phi: V_1 \rightarrow V_2$ such that $\mset{ \phi(e): e \in E_1} = E_2$.

\begin{defn}{\rm
Let $H=(V,E)$  be a hypergraph with $V=\{ v_1,\ldots,v_n \}$ and $E=\{ e_1,\ldots,e_m \}$, where $m \ne 0$. The incidence matrix of $H$ is an $n \times m$ matrix $M=(m_{ij})$ such that
$$m_{ij}=\left\{ \begin{array}{rl}
                    1 & \mbox{ if } v_i \in e_j \\
                    0 & \mbox{ otherwise}
                \end{array} \right. .
$$
}
\end{defn}
The following easy observation allows us to think of non-empty hypergraphs simply as 0-1 matrices.
\begin{lemma}
For any positive integers $m$ and $n$, let $M$ be an $n \times m$ 0-1 matrix. Then there exists a hypergraph $H=(V,E)$ with $|V|=n$ and $|E|=m$ such that $M$ is its incidence matrix.
\end{lemma}

Since we do have the notion of vertex adjacency for hypergraphs, we could also define (analogously to the adjacency matrix of a graph) the adjacency matrix of a hypergraph. However, in general, the adjacency matrix of a hypergraph, as opposed to a graph, will not contain full information about the hypergraph, and hence is of limited use.

Counting flags in two different ways, we easily obtain the following analogue of the Handshaking Lemma for graphs, and its immediate corollary.
\begin{lemma}
Let $H=(V,E)$ be a hypergraph with the flag (multi)set $F$. Then
$$ \sum_{v \in V} \deg(v)=|F|=\sum_{e \in E} |e|.$$
\end{lemma}

\begin{cor}
A hypergraph has an even number of vertices of odd degree if and only if it has an even number of edges of odd cardinality.
\end{cor}

\subsection{New hypergraphs from old}

In this section, we first give a comprehensive list of various types of useful substructures found in  hypergraphs. The following definitions are from \cite{Duc} and appear to be (so far) standard in hypergraph theory.

\begin{defn}{\rm \cite{Duc}
Let $H=(V,E)$ be a hypergraph.
\begin{enumerate}
\item  A hypergraph $H'=(V',E')$ is called a {\em subhypergraph} of $H$ if the incidence matrix of $H'$, after a suitable permutation of its rows and columns, is a submatrix of the incidence matrix of $H$.
\item For $V' \subseteq V$, a {\em subhypergraph of $H$ induced by $V'$} is the hypergraph with vertex set $V'$ and edge multiset $E'=\mset{ e \cap V': e \in E, e \cap V' \ne \emptyset }$.
\item For $E' \subseteq E$, the {\em partial hypergraph of $H$ determined by $E'$} is the hypergraph  $(\cup_{e \in E'} e, E')$.
\item For $V' \subseteq V$, the {\em trace of $H$ by $V'$}
is the partial hypergraph of $H$ determined by the multiset of edges $\mset{ e \in E: e \subseteq V' }$.
\end{enumerate}
}
\end{defn}

\begin{rem}{\rm For anyone trying to generalize graph-theoretic terms to hypergraphs, the above terms may be confusing or inadequate. To illustrate this issue, let $H=(V,E)$ be a 2-uniform hypergraph. Then $H$ can be thought of as a graph. However:
\begin{enumerate}
\item A subhypergraph of $H$ may contain edges of cardinality less than 2, and hence is not a subgraph of $H$.
\item Similarly, a subhypergraph of $H$ induced by $V' \subseteq V$ may contain edges of cardinality 1, and hence is not a subgraph of $H$.
\item The partial hypergraph of $H$ determined by $E' \subseteq E$ is just a subgraph of $H$ induced by the edge set $E'$.
\item The trace of $H$ by $V' \subseteq V$ is the subgraph of $H$ induced by the set of vertices $V''=\{ v \in V': v \in e \mbox{ for some } e \in E, e \subseteq V' \}$.
\end{enumerate}
Moreover, for a general hypergraph $H$:
\begin{enumerate}
\item Neither $H$ nor its subhypergraph can be empty.
\item A subhypergraph of $H$ induced by a set of isolated vertices of $H$ is not a subhypergraph, because it is empty.
\item A subhypergraph $H'$ of $H$ need not be simple even if $H$ is, and $H'$ may contain empty edges (unless it is an induced subhypergraph) or edges of cardinality 1 even if $H$ does not.
\item If $H'$ is the trace of $H$ by $V' \subseteq V$, then the vertex set of $H'$ need not be $V'$; it may be a proper subset of $V'$.
\end{enumerate}
}
\end{rem}

In view of these problems, we propose the following modified and additional definitions to be used in this paper.
\begin{defn}{\rm
Let $H=(V,E)$ be a hypergraph.
\begin{enumerate}
\item  A hypergraph $H'=(V',E')$ is called a {\em subhypergraph} of $H$ if $V' \subseteq V$ and either $E'=\emptyset$ or the incidence matrix of $H'$, after a suitable permutation of its rows and columns, is a submatrix of the incidence matrix of $H$. (Thus, every edge $e' \in E'$ is of the form $e \cap V'$ for some $e \in E$, and the corresponding mapping from $E'$ to $E$ is injective.)
\item A subhypergraph $H'=(V',E')$ of $H$ with $E'=\mset{ e \cap V': e \in E, e \cap V' \ne \emptyset }$ is said to be {\em induced by $V'$}.
\item If $|V| \ge 2$ and $v \in V$, then $H\b v$ will denote the subhypergraph of $H$ induced by $V-\{ v\}$, also called a {\em vertex-deleted subhypergraph} of $H$.
\item  A hypergraph $H'=(V',E')$ is called a {\em hypersubgraph} of $H$ if $V' \subseteq V$ and $E' \subseteq E$.
\item  A hypersubgraph $H'=(V',E')$ of $H$ is said to be {\em induced by $V'$},  denoted by $H[V']$, if $E'=\mset{ e \in E: e \subseteq V', e \ne \emptyset }$.
\item  A hypersubgraph $H'=(V',E')$ of $H$ is said to be {\em induced by $E'$}, denoted by $H[E']$, if $V'=\cup_{e \in E'} e$.
\item For $E' \subseteq E$ and $e \in E$, we write shortly $H - E'$ and $H-e$ for the hypersubgraphs $(V,E-E')$ and $(V,E-\mset{e})$, respectively. The hypersubgraph $H-e$ may also be called an {\em edge-deleted hypersubgraph}.
\item A hypersubgraph $H'=(V',E')$ of $H$ is called {\em spanning} if $V'=V$.
\item An {\em $r$-factor} of $H$ is a spanning $r$-regular hypersubgraph of $H$.
\end{enumerate}
}
\end{defn}

Observe that, informally speaking, the vertex-deleted subhypergraph $H \b v$ is obtained from $H$ by removing vertex $v$ from $V$ and from all edges of $H$, and then discarding the empty edges.

It is easy to see that every hypersubgraph of $H=(V,E)$ is also a subhypergraph of $H$, but not conversely. However, not every  hypersubgraph of $H$ induced by $V' \subseteq V$ is a subhypergraph of $H$ induced by $V'$.

Observe also that if $H$ is a 2-uniform hypergraph (and hence a loopless graph), its hypersubgraphs, vertex-subset-induced hypersubgraphs, edge-subset-induced hypersubgraphs, edge-deleted hypersubgraphs, spanning hypersubgraphs, and factors are precisely its subgraphs, vertex-subset-induced subgraphs, edge-subset-induced subgraphs, edge-deleted subgraphs, spanning subgraphs, and factors (in the graph-theoretic sense), respectively. However, its vertex-deleted subgraphs are obtained by deleting all singleton edges from its vertex-deleted subhypergraphs.

\begin{rem}{\rm In \cite{Vol}, a {\em subhypergraph} is defined as our hypersubgraph, a {\em partial hypergraph} as our spanning hypersubgraph, and a {\em subhypergraph  induced by a subset of vertices} as our vertex-set-induced hypersubgraph. We do, however, appreciate the more general definition of a subhypergraph from \cite{Duc}, and would like to make a distinction between subhypergraphs that are hypersubgraphs and those that are not.

Note also that in \cite{Vol},  edge deletion as defined above is called {\em weak} edge deletion, and weak vertex deletion is defined as our vertex deletion except that empty edges are not discarded. In addition, strong vertex and edge deletion are defined as follows. To {\em strongly delete} a vertex $v$ from a hypergraph $H=(V,E)$, we remove vertex $v$ from $V$ and remove all edges containing $v$ from $E$. To {\em strongly delete} an edge $e$ from $H$, we remove edge $e$ from $E$, as well as all vertices contained in $e$ from both $V$ and from all edges incident with them.
}
\end{rem}

Next, we define union and intersection of hypergraphs. The incidence function will be needed to make this definition precise.

\begin{defn}\label{def:U}{\rm
Let $H_1=(V_1,E_1)$ and $H_2=(V_2,E_2)$ be hypergraphs with incidence functions $\psi_1$ and $\psi_2$, respectively, such that $\psi_1 \vert_{E_1 \cap E_2}=\psi_2 \vert_{E_1 \cap E_2}$. Let $\psi: E_1 \cup E_2 \rightarrow 2^{V_1 \cup V_2}$ be defined by
$$\psi(e)=\left\{ \begin{array}{rl}
                    \psi_1(e) & \mbox{ if } e \in E_1 \\
                    \psi_2(e) & \mbox{ if } e \in E_2 \\
                    \end{array} \right. .$$
The {\em union} of $H_1$ and $H_2$, denoted $H_1 \cup H_2$, is then defined as the hypergraph $(V_1 \cup V_2,E_1 \cup E_2)$ with the incidence function $\psi$, and the {\em intersection} of $H_1$ and $H_2$, denoted $H_1 \cap H_2$, as the hypergraph $(V_1 \cap V_2,E_1 \cap E_2)$ with the incidence function $\psi \vert_{E_1 \cap E_2}$.

If a hypergraph $H$ is an edge-disjoint union of hypegraphs $H_1$ and $H_2$ (that is, $H=H_1 \cup H_2$ with $E_1 \cap E_2 =\emptyset$), then we say that $H$ {\em decomposes} into $H_1$ and $H_2$, and write $H=H_1 \oplus H_2$.
}
\end{defn}

The last operation we shall introduce is the dual, clearly inherited from designs.

\begin{defn}{\rm
The {\em dual} of a non-empty hypergraph $H$ is a hypergraph $H^T$ whose incidence matrix is the transpose of the incidence matrix of $H$.
}
\end{defn}

To obtain the dual $H^T=(E^T,V^T)$ of a hypergraph $H=(V,E)$, we label the edges of $H$ as $e_1,\ldots,e_m$ (with distinct parallel edges receiving distinct labels). Then let $E^T=\{ e_1,\ldots,e_m \}$ and $V^T=\mset{ v^T: v \in V}$, where $v^T=\{ e \in E^T: v \in e \}$ for all $v \in V$. Observe that $(v,e) \in F(H)$ if and only if $(e,v^T) \in F(H^T)$. Hence $(H^T)^T=H$.

\begin{lemma}\label{lem:delT}
Let $H=(V,E)$ be a non-empty hypergraph with the dual $H^T=(E^T,V^T)$, and let $v\in V$ and $e \in E$. Then:
\begin{enumerate}
\item $\deg_H(v)=|v^T|$.
\item $v$ is an isolated vertex (pendant vertex) in $H$ if and only if $v^T$ is an empty edge (singleton edge, respectively) in $H^T$.
\item If $|V|\ge 2$, $H$ has no empty edges, and $\{ v \} \not\in E$, then $(H \b v)^T=H^T - v^T$.
\item If $|E| \ge 2$, $H$ has no isolated vertices, and $e$ contains no pendant vertices, then $(H-e)^T=H^T \b e$.
\end{enumerate}
\end{lemma}

\begin{proof}
The first two statements of the lemma follow straight from the definition of vertex degree.

To see the third statement, assume that $|V|\ge 2$, $H$ has no empty edges, and  $\{ v \} \not\in E$. Now $H \b v$ is obtained from $H$ by  deleting vertex $v$, deleting all flags containing $v$ from $F(H)$, and discarding all resulting empty edges. Hence $(H \b v)^T$ is obtained from $H^T$ by deleting edge $v^T$, deleting all flags containing $v^T$ from $F(H^T)$, and discarding all resulting isolated vertices. However, any such isolated vertex would in $H^T$ correspond either to an isolated vertex or a pendant vertex incident only with the edge $v^T$. This would imply existence of an empty edge or an edge $\{ v \}$ in $H$, a contradiction. Hence $(H \b v)^T$ was obtained from $H^T$ just by deleting edge $v^T$ and all flags containing $v^T$; that is, $(H \b v)^T=H^T-v^T$.

To prove the fourth statement, assume that $|E| \ge 2$, $H$ has no isolated vertices, and $e$ contains no pendant vertices. Recall that $H-e$ is obtained from $H$, and similarly $(H-e)^T$ from $H^T$, by  deleting $e$ and all flags containing $e$. This operation on $H^T$ is exactly vertex deletion provided that $(H-e)^T$ has no empty edges. Now an empty edge in $(H-e)^T$ corresponds to an isolated vertex in $H-e$, and hence in $H$, it corresponds either to an isolated vertex or a pendant vertex incident with $e$. However, by assumption, $H$ does not have such vertices. We conclude that $(H-e)^T=H^T \b e$ as claimed.
\end{proof}

\subsection{Graphs associated with a hypergraph}

A hypergraph is, of course, an incidence structure, and hence can be represented with an incidence graph (to be defined below). This representation retains complete information about the hypergraph, and thus allows us to translate problems about hypergraphs into problems about graphs --- a much better explored territory.

\begin{defn}{\rm
Let  $H=(V,E)$ be a hypergraph with incidence function $\psi$. The {\em incidence graph} $\G(H)$ of $H$ is the graph $\G(H)=(V_{G},E_G)$ with $V_G=V \cup E$ and $E_G=\{ ve: v\in V, e \in E, v \in \psi(e)\}$.
}
\end{defn}

Observe that the incidence graph $\G(H)$ of a hypergraph $H=(V,E)$ with $E \ne \emptyset$ is a bipartite simple graph with bipartition $\{ V, E \}$. We shall call a vertex $x$ of $\G(H)$ a {\em v-vertex} if $x \in V$, and an {\em e-vertex} if $x \in E$. Note that the edge set of $\G(H)$ can be identified with the flag (multi)set $F(H)$; that is, $E_G=\{ ve: (v,e) \in F(H)\}$.

The following is an easy observation, hence the proof is left to the reader.
\begin{lemma}\label{lem:isom}
Let $H=(V,E)$ be a non-empty hypergraph and $H^T=(E^T,V^T)$ its dual. The incidence graphs $\G(H)$ and $\G(H^T)$ are isomorphic with an isomorphism $\phi: V \cup E \rightarrow E^T \cup V^T$ defined by $\phi(e)=e$ for all $e \in E$, and $\phi(v)=v^T$ for all $v \in V$.
\end{lemma}

Next, we outline the relationship between subhypergraphs of a hypergraph and the subgraphs of its incidence graph. The proof of this lemma is straightforward and hence omitted.

\begin{lemma}\label{lem:sub-inc}
Let $H=(V,E)$ be a hypergraph and $H'=(V',E')$ a subhypergraph of $H$. Then:
\begin{enumerate}
\item $\G(H')$ is the subgraph of $\G(H)$ induced by the vertex set $V' \cup E'$.
\item If $H'$ is a hypersubgraph of $H$, then in addition, $\deg_{\G(H')}(e)=\deg_{\G(H)}(e)=|e|$ for all $e \in E'$.
\end{enumerate}
Conversely, take a subgraph $G'$ of $\G(H)$. Then:
\begin{enumerate}
\item $V(G')=V' \cup E'$ for some $V' \subseteq V$ and $E' \subseteq E$, and $E(G') \subseteq \{ ve: v \in V', e \in E', v \in e\}$.
\item $G'$ is the incidence graph of a subhypergraph of $H$ if and only if $V' \ne \emptyset$ and for all $e \in E'$ we have $\{ ve: v \in e \cap V' \} \subseteq E(G')$.
\item $G'$ is the incidence graph of a hypersubgraph of $H$ if and only if $V' \ne \emptyset$ and $\deg_{G'}(e)=\deg_{\G(H)}(e)=|e|$ for all $e \in E'$.
\end{enumerate}
\end{lemma}

In the following lemma, we determine the incidence graphs of vertex-deleted subhypergraphs and  edge-deleted hypersubgraphs.
\begin{lemma}\label{lem:delete}
Let $H=(V,E)$ be a hypergraph. Then:
\begin{enumerate}
\item For all $e \in E$, we have $\G(H-e)=\G(H) \b e$.
\item  If $|V| \ge 2$, $H$ has no empty edges, and $v \in V$ is such that $\{ v \} \not\in E$, then $\G(H \b v)=\G(H) \b v$.
\end{enumerate}
\end{lemma}

\begin{proof}
\begin{enumerate}
\item Recall that $H-e$ is obtained from $H$ by deleting $e$ from $E$, thus also destroying all flags containing $e$. This is equivalent to deleting $e$ from the vertex set of $\G(H)$, as well as all edges of $\G(H)$ incident with $e$, which results in the vertex-deleted subgraph $\G(H) \b e$.
\item Now $H \b v$ is obtained from $H$ by deleting $v$ from $V$ and from all edges containing $v$, and then discarding all resulting empty edges. However, if $H$ has no empty edges and $\{ v \} \not\in E$, then there are no empty edges to discard, and so this operation is equivalent to deleting $v$ from the vertex set of $\G(H)$ and deleting all edges of $\G(H)$ incident with $v$, resulting in the vertex-deleted subgraph $\G(H) \b v$. Hence $\G(H) \b v=\G(H \b v)$.
\end{enumerate}
\vspace{-1.6cm}
\end{proof}

Another graph associated with a hypergraph that can be useful --- although it does not contain full information about the hypergraph --- is the line graph (also called the intersection graph).

\begin{defn}{\rm
The {\em line graph} (or {\em intersection graph}) of the hypergraph $H=(V,E)$, denoted $\L(H)$, is the graph with vertex set $E$ and edge set $\{ e e': e, e'\in E, e \ne e', e \cap e' \ne \emptyset \}$.

More generally, for any positive integer $\ell$, we define the {\em level-$\ell$ line graph} of the hypergraph $H=(V,E)$, denoted $\L_{\ell}(H)$, as the graph with vertex set $E$ and edge set $\{ e e': e, e'\in E, e \ne e', |e \cap e'| \ge \ell \}$.
}
\end{defn}

\section{Connection in Hypergraphs}

\subsection{Walks, trails, paths, cycles}

In this section, we would like to systematically generalize the standard graph-theoretic notions of walks, trails, paths, and cycles to hypergraphs. In this context, we need to distinguish between distinct parallel edges, hence the original definition of a hypergraph that includes the incidence function will be used.

\begin{defn}{\rm
Let $H=(V,E)$ be a hypergraph with incidence function $\psi$, let $u,v \in V$, and let $k \ge 0$ be an integer. A {\em $(u,v)$-walk of length $k$}  in $H$ is a sequence $v_0 e_1 v_1 e_2 v_2 \ldots v_{k-1} e_k v_k$ of vertices and edges (possibly repeated) such that $v_0,v_1,\ldots,v_k \in V$, $e_1,\ldots,e_k \in E$, $v_0=u$, $v_k=v$, and for all $i=1,2,\ldots, k$, the vertices $v_{i-1}$ and $v_i$ are adjacent in $H$ via the edge $e_i$.

If $W=v_0 e_1 v_1 e_2 v_2 \ldots v_{k-1} e_k v_k$ is a walk in $H$, then vertices  $v_0$ and $v_k$ are called the {\em endpoints} of $W$, and $v_1,\ldots,v_{k-1}$ are the {\em internal vertices} of $W$.

Furthermore, vertices $v_0,v_1,\ldots,v_k$ are called the {\em anchors} of $W$, and any vertex $u \in e_i$, for some $i \in \{ 1,2,\ldots, k\}$, that is not an anchor of $W$ is called a {\em floater} of $W$. We write $V_a(W)$, $V_f(W)$, and $E(W)$ to denote the sets of anchors, floaters, and edges of a walk $W$.
}
\end{defn}

Observe that since adjacent vertices are by definition distinct, no two consecutive vertices in a walk are the same. Note that the edge set $E(W)$ of a walk $W$ may contain distinct parallel edges.

Recall that a trail in a graph is a walk with no repeated edges. For a walk in a graph, having no repeated edges is necessary and sufficient for having no repeated flags; in a hypergraph, only sufficiency holds. This observation suggests two possible ways to define a trail.

\begin{defn}{\rm
Let $W=v_0 e_1 v_1 e_2 v_2 \ldots v_{k-1} e_k v_k$ be a walk in a hypergraph $H=(V,E)$ with incidence function $\psi$.
\begin{enumerate}
\item If the anchor flags $(v_0,e_1),(v_1, e_1),(v_1,e_2),\ldots,(v_{k-1},e_k),(v_k,e_k)$ are pairwise distinct, then $W$ is called a {\em trail}.
\item If the edges $e_1,\ldots,e_k$ are pairwise distinct, then $W$ is called a {\em strict trail}.
\item If the anchor flags $(v_0,e_1),(v_1, e_1),(v_1,e_2),\ldots,(v_{k-1},e_k),(v_k,e_k)$ and the vertices  $v_0,v_1,$ $\ldots,v_k$ are pairwise distinct (but the edges need not be), then $W$ is called a {\em pseudo path}.
\item If both the vertices  $v_0,v_1,\ldots,v_k$ and the edges $e_1,\ldots,e_k$ are pairwise distinct, then $W$ is called a {\em path}.
\end{enumerate}
}
\end{defn}
We emphasize that in the above definitions, ``distinct'' should be understood in the strict sense; that is, parallel edges need not be distinct.

We extend the above definitions to closed walks in the usual way.
\begin{defn}{\rm
Let $W=v_0 e_1 v_1 e_2 v_2 \ldots v_{k-1} e_k v_k$ be a walk  in a hypergraph $H=(V,E)$ with incidence function $\psi$. If  $k \ge 2$ and $v_0=v_k$, then $W$ is called a {\em closed walk}. Moreover:
\begin{enumerate}
\item If $W$ is a trail (strict trail), then it is called a {\em closed trail} ({\em closed strict trail}, respectively).
\item  If $W$ is a closed trail and the vertices  $v_0,v_1,\ldots,v_{k-1}$ are pairwise distinct (but the edges need not be), then $W$ is called a {\em pseudo cycle}.
\item If  the vertices  $v_0,v_1,\ldots,v_{k-1}$ and the edges $e_1,\ldots,e_k$ are pairwise distinct, then $W$ is called a {\em cycle}.
\end{enumerate}
}
\end{defn}
From the above definitions, the following observations are immediate.
\begin{lemma}
Let $W$ be a walk in a hypergraph $H$. Then:
\begin{enumerate}
\item If $W$ is a trail, then no two consecutive edges in $W$ are the same (including the last and the first edge if $W$ is a closed trail).
\item If $W$ is a (closed) strict trail, then it is a (closed) trail.
\item If $W$ is a pseudo path (pseudo cycle), then it is a trail (closed trail, respectively), but not necessarily a strict trail (closed strict trail, respectively).
\item If $W$ is a path (cycle), then it is both a pseudo path (pseudo cycle, respectively) and a strict trail (closed strict trail, respectively).
\end{enumerate}
\end{lemma}

In a graph, a path or cycle can be identified with the corresponding subgraph (also called path or cycle, respectively). This is not the case in hypergraphs. First, we note that there are (at least) two ways to define a subhypergraph associated with a path or cycle. We define these more generally for walks.

\begin{defn}\label{def:HH'}{\rm
Let $W$ be a walk in a hypergraph $H=(V,E)$. Define the hypersubgraph  $\H(W)$ and a subhypergraph $\H'(W)$ of $H$ associated with the walk $W$ as follows:
$$\H(W)=(V_a(W)\cup V_f(W), E(W))$$
and
$$\H'(W)=(V_a(W), \mset{ e \cap V_a(W): e \in E(W)}).$$
That is, $\H'(W)$ is the subhypergraph  of $\H(W)$ induced by the set of anchor vertices $V_a(W)$.
}
\end{defn}

Second, we observe that, even when $W$ is a path or a cycle, not much can be said about the degrees of the vertices in the associated subhypergraphs $\H(W)$ and $\H'(W)$. Thus, unlike in graphs, we can not use a path (cycle) $W$ (as a sequence of vertices and edges) and its associated subhypergraphs $\H(W)$ and $\H'(W)$ interchangeably.

The following lemma will justify the terminology introduced in this section.

\begin{lemma}\label{lem:W-W_G}
Let $H=(V,E)$  be a hypergraph and $G=\G(H)$ its incidence graph.  Let $v_i \in V$ for $i=0,1,\ldots,k$, and $e_i \in E$ for $i=1,\ldots,k$, and let $W=v_0 e_1 v_1 e_2 v_2 \ldots v_{k-1} e_k v_k$ be an alternating sequence of vertices and edges of $H$. Denote the corresponding sequence of vertices in $G$ by $W_G$. Then the following hold:
\begin{enumerate}
\item $W$ is a (closed) walk in $H$ if and only if $W_G$ is a (closed) walk in $G$ with no two consecutive v-vertices the same.

\item $W$ is a trail (path, cycle) in $H$ if and only if $W_G$ is a trail (path, cycle, respectively) in $G$.

\item $W$ is a strict trail in $H$ if and only if $W_G$ is a trail in $G$ that visits every $e \in E$ at most once.

\item $W$ is a pseudo path (pseudo cycle)  in $H$ if and only if $W_G$ is a trail (closed trail, respectively) in $G$ that visits every $v \in V$ at most once.
\end{enumerate}
\end{lemma}

\begin{proof}
\begin{enumerate}
\item If $W$ is a walk in $H$, then any two consecutive elements of the sequence $W$ are incident in $H$, and hence the corresponding vertices are adjacent in $G$. Thus $W_G$ is a walk in $G$. Moreover, no two consecutive vertices in $W$ are the same, whence not two consecutive v-vertices in $W_G$ are the same. The converse is shown similarly. Clearly $W$ is closed if and only if $W_G$ is.

    Observe that the anchor vertices and the edges of $W$ correspond to the v-vertices and e-vertices of $W_G$, respectively, and the anchor flags of $W$ correspond to the edges of $W_G$.

\item If $W$ is a trail in $H$, then $W$ is a walk with no repeated anchor flags;  hence $W_G$ is a walk in $G$ with no repeated edges, that is, a trail.
    Conversely, if $W_G$ is a trail in $G$, then it is a walk with no repeated edges, and hence no two identical consecutive v-vertices. It follows that $W$ is a walk in $H$ with no repeated anchor flags, that is, a trail.

    Similarly, if $W$ is a path (cycle)  in $H$, then $W$ is a walk with no repeated edges and no repeated vertices (except the endpoints for a cycle). Hence $W_G$ is a walk in $G$ with no repeated vertices (except the endpoints for a cycle), that is, a path (cycle, respectively). The converse is shown similarly.

\item If $W$ is a strict trail in $H$, then it is a trail with no repeated edges. Hence $W_G$ is a trail in $G$ with no repeated e-vertices. The converse is shown similarly.

\item If $W$ is a pseudo path (pseudo cycle) in $H$, then it is a trail with no repeated vertices (except the endpoints for a pseudo cycle). Hence $W_G$ is a trail in $G$ with no repeated v-vertices (except the endpoints for a pseudo cycle). The converse is similar.
\end{enumerate}
\vspace{-1.6cm}
\end{proof}

The next observations are easy to see, hence the proof is omitted.

\begin{lemma}\label{lem:W^T}
Let $H=(V,E)$  be a non-empty hypergraph and $H^T=(E^T,V^T)$ its dual.  Let $v_i \in V$ for $i=0,1,\ldots,k-1$, and $e_i \in E$ for $i=0,1,\ldots,k-1$, and let $W=v_0 e_0 v_1 e_1 v_2 \ldots v_{k-1} e_{k-1} v_0$ be a closed walk in $H$. Denote $W^T=e_0 v_1^T e_1 v_2^T \ldots v_{k-1}^T e_{k-1} v_0^T e_0$, where for each vertex $v_i$  of $H$, the symbol $v_i^T$ denotes the corresponding edge  in $H^T$. Then the following hold:
\begin{enumerate}
\item If $e_i \ne e_{i+1}$ for all $i \in \ZZ_k$, then $W^T$ is a closed walk in $H^T$.
\item If $W$ is a closed trail (cycle) in $H$, then $W^T$ is a closed trail (cycle, respectively) in $H^T$.
\item If $W$ is a strict closed trail in $H$, then $W^T$ is a pseudo cycle in $H^T$.
\item If $W$ is a pseudo cycle in $H$, then $W^T$ is a strict closed trail in $H^T$.
\end{enumerate}
\end{lemma}

To complete this section, we define concatenation of walks in the usual way.

\begin{defn}{\rm
Let $W=v_0 e_1 v_1  \ldots e_k v_k$ and
$W'=v_k e_{k+1} v_{k+1}  \ldots e_{\ell} v_{\ell}$, for $0 \le k \le \ell$, be two walks in a hypergraph $H=(V,E)$.  The {\em concatenation} of $W$ and $W'$ is the walk $WW'=v_0 e_1 v_1 \ldots e_k v_k e_{k+1} v_{k+1}  \ldots e_{\ell} v_{\ell}$.
}
\end{defn}

\subsection{Connected hypergraphs}

Connected hypergraphs are defined analogously to connected graphs, using existence of walks (or equivalently, existence of paths) between every pair of vertices. The main result of this section is the observation that a hypergraph (without empty edges) is connected if and only if its incidence graph is connected. The reader will observe that existence of empty edges in a hypergraph does not affect its connectivity; however, it does affect the connectivity of the incidence graph.

\begin{defn}{\rm
Let  $H=(V,E)$ be a hypergraph. Vertices $u,v \in V$ are said to be {\em connected} in $H$ if there exists a $(u,v)$-walk in $H$. The hypergraph $H$ is said to be {\em connected} if every pair of distinct vertices are connected in $H$.
}
\end{defn}

\begin{lemma}
Let  $H=(V,E)$ be a hypergraph, and $u,v \in V$. There exists a $(u,v)$-walk in $H$ if and only if there exists a $(u,v)$-path.
\end{lemma}

\begin{proof}
Suppose $H$ has a $(u,v)$-walk. By Lemma~\ref{lem:W-W_G}, it corresponds to a $(u,v)$-walk in the incidence graph $\G(H)$, and by a classical result in graph theory, existence of a $(u,v)$-walk in a graph guarantees existence of a $(u,v)$-path. Finally, by Lemma~\ref{lem:W-W_G}, a $(u,v)$-path  in $\G(H)$ (since $u,v
\in V$) corresponds to a $(u,v)$-path in $H$.

The converse obviously holds by definition.
\end{proof}

It is clear that vertex connection in a hypergraph $H=(V,E)$ is an equivalence relation on the set $V$. Hence the following definition makes sense.

\begin{defn}{\rm
Let  $H=(V,E)$ be a hypergraph, and let $V' \subseteq V$ be an equivalence class with respect to vertex connection. The hypersubgraph of $H$ induced by $V'$ is called a {\em connected component} of $H$. We denote the number of connected components of $H$ by $\omega(H)$.
}
\end{defn}

Observe that, by the definition of a vertex-subset-induced hypersubgraph, the connected components of a hypergraph have no empty edges. Alternatively, the connected components of $H$ can be defined as the maximal connected hypersubgraphs of $H$ that have no empty edges. It is easy to see that for a hypergraph $H=(V,E)$ with the multiset of empty edges denoted $E_0$, the hypersubgraph $H-E_0$ decomposes into the connected components of $H$.

\begin{theo}\label{the:conn}
Let $H=(V,E)$ be a hypergraph  without empty edges. Then $H$ is connected if and only if its incidence graph $G=\G(H)$ is connected.
\end{theo}

\begin{proof}
Assume $H$ is connected. Take any two vertices $x,y$ of $G$. If $x$ and $y$ are both v-vertices, then there exists an $(x,y)$-walk in $H$, and hence, by Lemma~\ref{lem:W-W_G}, an $(x,y)$-walk in $G$. If $x$ is an e-vertex and $y$ is a v-vertex in $G$ , then $x$ is a non-empty edge in $H$. Choose any $v \in x$. Since $H$ is connected, it possesses a $(v,y)$-walk $W$. Then $xW$ is an $(x,y)$-walk in $G$. The remaining case $x,y \in E$ is handled similarly. We conclude that $G$ is connected.

Assume $G$ is connected. Take any two vertices $u,v$ of $H$. Then there exists $(u,v)$-path in $G$, and hence by Lemma~\ref{lem:W-W_G}, a $(u,v)$-path in $H$. Therefore $H$ is connected.
\end{proof}

\begin{cor}\label{cor:conn}
Let $H$ be a hypergraph and  $G=\G(H)$ its incidence graph. Then:
\begin{enumerate}
\item If $H'$ is a connected component of $H$, then $\G(H')$ is a connected component of $G$.
\item If $G'$ is a connected component of $G$ with at least one v-vertex, then there exists a connected component $H'$ of $H$ such that $G'=\G(H')$.
\item If $H$ has no empty edges, then there is a one-to-one correspondence between connected components of $H$ and connected components of its incidence graph.
\end{enumerate}
\end{cor}

\begin{proof}
\begin{enumerate}
\item Let $H'$ be a connected component of $H$, and let $G'=\G(H')$. Since $H'$ has no empty edges by definition,  $G'$ is connected by Theorem~\ref{the:conn}. Let $G''$ be the connected component of $G$ containing $G'$ as a subgraph. Then $G''$ contains v-vertices and  $\deg_{G''}(e)=\deg_{G}(e)$ for all e-vertices $e$ of $G''$, and so by Lemma~\ref{lem:sub-inc},  $G''=\G(H'')$ for some hypersubgraph $H''$ of $H$. Since $G''$ is connected and the incidence graph of a hypergraph, it has no isolated e-vertices. Hence $H''$ has no empty edges, and so by Theorem~\ref{the:conn}, $H''$ is connected since $G''$ is. Now $H'$ is a maximal connected hypersubgraph of $H$ without empty edges, and a hypersubgraph of a connected hypersubgraph $H''$ without empty edges; it must be that $H''=H'$. Consequently, $G'=G''$ and so $G'$ is indeed a connected component of $G$.
\item Let $G'$ be a connected component of $G$ with at least one v-vertex. Then $\deg_{G'}(e)=\deg_{G}(e)$ for all e-vertices $e$ of $G'$, and so by Lemma~\ref{lem:sub-inc},  $G'=\G(H')$ for some hypersubgraph $H'$ of $H$. Since $G'$ is connected and the incidence graph of a hypergraph, it has no isolated e-vertices; hence $H'$ has no empty edges. Thus, by Theorem~\ref{the:conn}, $H'$ is connected since $G'$ is. Let $H''$ be the connected component of $H$ containing $H'$, and  $G''=\G(H'')$. Again by Theorem~\ref{the:conn}, $G''$ is connected, and hence $G''=G'$ by the maximality of $G'$. It follows that $H'=H''$, so indeed $G'=\G(H')$, where $H'$ is a connected component of $H$.
\item Since $H$ has no empty edges, every connected component of $G$ has at least one v-vertex. The conclusion now follows directly from the first two statements of the corollary.
\end{enumerate}
\vspace{-1.7cm}
\end{proof}

\begin{cor}\label{cor:omega}
Let $H$ be a hypergraph without empty edges and $G=\G(H)$ its incidence graph. Then:
\begin{enumerate}
\item $\omega(H)=\omega(G)$.
\item If $H$ is non-empty and has no isolated vertices, and $H^T$ is its dual, then $\omega(H)=\omega(H^T)$.
\end{enumerate}
\end{cor}

\begin{proof}
\begin{enumerate}
\item Since $H$ has no empty edges, by Corollary~\ref{cor:conn} there is a one-to-one correspondence between the connected components of $H$ and $G$. Therefore, $\omega(H)=\omega(G)$.
\item Assume $H$ is non-empty and has no isolated vertices. Then $H^T$ is well defined and has no empty edges, and so $\omega(H^T)=\omega(\G(H^T))$ by the first statement. Since by Lemma~\ref{lem:isom} a hypergraph and its dual have isomorphic incidence graphs, it follows that $\omega(H^T)=\omega(\G(H^T))=\omega(G)=\omega(H).$
\end{enumerate}
\vspace{-10mm}
\end{proof}

\subsection{Cut edges and cut vertices}

In this section, we define cut edges and cut vertices in a hypergraph analogously to those in a graph. The existence of cut edges and cut vertices is one of the first measures of strength of connectivity of a connected (hyper)graph. In hypergraphs, however, we must consider two distinct types of cut edges.

\begin{defn}{\rm
A {\em cut edge} in a hypergraph $H=(V,E)$ is an edge $e \in E$ such that $\omega(H-e) > \omega(H)$.
}
\end{defn}

\begin{lemma}\label{lem:cut-edge}
Let $e$ be a cut edge in a hypergraph  $H=(V,E)$. Then $$\omega(H) < \omega(H-e) \le \omega(H)+|e|-1.$$
\end{lemma}

\begin{proof}
The inequality on the left follows straight from the definiton of a cut edge. To see the inequality on the right, first observe that $e$ is not empty. Let  $H_1,\ldots,H_k$ be the connected components of $H-e$ whose vertex sets intersect $e$. Since $e$ has at least one vertex in common with each $V(H_i)$, we have $|e| \ge k$. Hence $\omega(H-e)= \omega(H)+k-1 \le \omega(H)+|e|-1$.
\end{proof}

\begin{defn}{\rm
A cut edge $e$ of a hypergraph $H$ is called {\em strong} if $\omega(H-e) = \omega(H)+|e|-1$, and {\em weak} otherwise.
}
\end{defn}

Observe that a cut edge has cardinality at least two, and that any cut edge of cardinality two (and hence any cut edge in a simple graph) is necessarily strong.

Recall that an edge of a graph is a cut edge if and only if appears in no cycle. We shall now show that an analogous statement holds for hypergraphs if we replace ``cut edge'' with ``strong cut edge''.

\begin{theo}\label{the:strong-cut}
Let $e$ be an edge in a connected hypergraph  $H=(V,E)$. The following are equivalent:
\begin{enumerate}
\item $e$ is a strong cut edge, that is, $\omega(H-e)=|e|$.
\item $e$ contains exactly one vertex from each connected component of $H-e$.
\item $e$ lies in no cycle of $H$.
\end{enumerate}
\end{theo}

\begin{proof}
(1) $\Rightarrow$ (2): Let $e$ be a strong cut edge of $H$. Since $H$ is connected, the edge $e$ must have at least one vertex in each  connected component of $H-e$. Since there are $|e|$ connected components of $H-e$, the edge $e$ must have exactly one vertex in each of them.

(2) $\Rightarrow$ (1): Assume $e$ contains exactly one vertex from each connected component of $H-e$. Then clearly $\omega(H-e)=|e|$.

(2) $\Rightarrow$ (3): Assume $e$ contains exactly one vertex from each connected component of $H-e$, and suppose $e$ lies in a cycle $C=v_0 e_1 v_1 e_2 v_2 \ldots v_{k-1} e v_0$ of $H$. Then $v_0 e_1 v_1 e_2 v_2 \ldots v_{k-1}$ is a path in $H-e$, and so $v_0$ and $v_{k-1}$ are two vertices of $e$ in the same connected component of $H-e$, a contradiction. Hence $e$ lies in no cycle of $H$.

(3) $\Rightarrow$ (2): Assume $e$ lies in no cycle of $H$. Since $H$ is connected, the edge $e$ must contain at least one vertex from each connected component of $H-e$. Suppose $e$ contains two vertices $u$ and $v$ in the same connected component $H'$ of $H-e$. Then $H'$ contains a $(u,v)$-path $P$, and $Pv e u$ is a cycle in $H$ that contains $e$, a contradiction. Hence $e$ possesses exactly one vertex from each connected component of $H-e$.
\end{proof}

The above theorem can be easily generalized to all (posssibly disconnected) hypergraphs as follows.

\begin{cor}\label{cor:strong-cut}
Let $e$ be an edge in a hypergraph  $H=(V,E)$. The following are equivalent:
\begin{enumerate}
\item $e$ is a strong cut edge, that is, $\omega(H-e) = \omega(H)+|e|-1$.
\item $e$ contains exactly one vertex from each connected component of $H-e$ that it intersects.
\item $e$ lies in no cycle of $H$.
\end{enumerate}
\end{cor}

We know that an even graph has no cut edges; in other words, every edge of an even graph (that is, a graph with no odd-degree vertices) lies in a cycle. This statement is false for hypergraphs, as the example below demonstrates. In the following two theorems, however, we present two generalizations to hypergraphs that do hold.

\begin{counter}{\rm
For every even $n\ge 2$, define a hypergraph $H=(V,E)$ as follows. Let $V=\{ v_i: i=1,\ldots,2n \}$ and $E=\{ e_i:   i=1,\ldots,2n \}$, and let $F(H)=\{ (v_i,e_j): i,j=1,\ldots,n \} \cup \{ (v_i,e_j): i,j=n+1,\ldots,2n \} \cup \{ (v_1,e_{n+1}) \} - \{ (v_1,e_1)\}$. Then every vertex in $H$ has degree $n$, which is even, but $e_{n+1}$ is a cut edge in $H$.
}
\end{counter}

\begin{theo}
Let $H=(V,E)$ be a $k$-uniform hypergraph such that $\deg_H(u)\equiv 0 \pmod{k}$ for every vertex $u$ of $H$. Then $H$ has no cut edges.
\end{theo}

\begin{proof}
Suppose $e$ is a cut edge of $H$, and let $H_1=(V_1,E_1)$ be a connected component of $H-e$ that contains a vertex of $e$. Furthermore, let $r=|e \cap V_1|$. Then $1 \le r \le k-1$, and so $\sum_{v\in V_1} \deg_{H_1}(v) \equiv k|V_1|-r\not\equiv 0 \pmod{k}$. However, $\sum_{v\in V_1} \deg_{H_1}(v)=\sum_{f \in E_1}|f|=k|E_1|$, a contradiction. Hence $H$ cannot have cut edges.
\end{proof}

\begin{theo}
Let $H=(V,E)$ be a hypergraph such that the degree of each vertex and the cardinality of each edge are even. If $e$ is a cut edge of $H$, then every connected component of $H-e$ contains an even number of vertices of $e$. In particular, $H$ has no strong cut edges.
\end{theo}

\begin{proof}
Suppose $e$ is a cut edge of $H$, and let $H_1=(V_1,E_1)$ be any connected component of $H-e$. Furthermore, let $r=|e \cap V_1|$. Then  $\sum_{v\in V_1} \deg_{H_1}(v)=(\sum_{v\in V_1} \deg_{H}(v))-r=\sum_{f \in E_1} |f|$. Since $\sum_{v\in V_1} \deg_{H}(v)$ and $\sum_{f \in E_1} |f|$ are both even, so is $r$. Thus $e$ intersects every connected component in an even number of vertices, and hence by Corollary~\ref{cor:strong-cut} cannot be a strong cut edge.
\end{proof}

We now turn our attention to cut vertices. Recall that the vertex-deleted subhypergraph $H \b v$ is obtained from $H$ by deleting $v$ from the vertex set, as well as from all edges containing $v$, and then discarding any resulting empty edges.

\begin{defn}{\rm
A {\em cut vertex} in a hypergraph $H=(V,E)$ with $|V|\ge 2$ is a vertex $v \in V$ such that $\omega(H \b v) > \omega(H)$.
}
\end{defn}

Before we can prove a result similar to Lemma~\ref{lem:cut-edge} for cut vertices, we need to examine the relationship between cut vertices and cut edges of a hypergraph and its dual, as well as the relationship between cut vertices and cut edges of a hypergraph and cut vertices of its incidence graph.

\begin{theo}\label{the:cutG}
Let $H=(V,E)$ be a hypergraph without empty edges, and $G=\G(H)$ be its incidence graph.
\begin{enumerate}
\item Take any $e \in E$. Then $e$ is a cut edge of $H$ if and only if it is a cut vertex of $G$.
\item Let $|V| \ge 2$ and take any $v \in V$ such that $\{ v \} \not\in E$. Then $v$ is a cut vertex of $H$ if and only if it is a cut vertex of $G$.
\end{enumerate}
\end{theo}

\begin{proof}
\begin{enumerate}
\item By Lemma~\ref{lem:delete}, we have $\G(H-e)=G \b e$. Since $H$, and hence $H-e$, has no empty edges, Corollary~\ref{cor:omega} tells us that $\omega(H)=\omega(G)$ and $\omega(H-e)=\omega(\G(H-e))$.  Hence $\omega(H-e)=\omega(G\b e)$. Thus $\omega(H-e)-\omega(H)=\omega(G\b e)-\omega(G)$, and it follows that $e$ is a cut edge of $H$ if and only if it is a cut vertex of $G$.
\item Since $H$ has no empty edges and $\{ v \} \not\in E$, Lemma~\ref{lem:delete} shows that $\G(H \b v)=G \b v$. Since $H$ and $H \b v$ have no empty edges, Corollary~\ref{cor:omega} gives $\omega(H)=\omega(G)$ and $\omega(H \b v)=\omega(\G(H \b v))$, respectively. Hence $\omega(H \b v)-\omega(H)=\omega(G \b v)-\omega(G)$, and $v$ is a cut vertex of $H$ if and only if it is a cut vertex of $G$.
\end{enumerate}
\vspace{-1.6cm}
\end{proof}

In the next corollary, recall that we denote the dual of a hypergraph $H=(V,E)$ by $H^T=(E^T,V^T)$, where $E^T$ is the set of labels for the edges in $E$, $V^T=\{ v^T: v\in V \}$, and $v^T=\{ e \in E^T: v \in e \}$ for all $v \in V$.

\begin{cor}\label{cor:cuts}
Let $H=(V,E)$ be a non-empty hypergraph  with neither empty edges nor isolated vertices, and let $H^T$ be its dual.
\begin{enumerate}
\item Let $|E| \ge 2$ and let $e \in E$ be an edge without pendant vertices. Then $e$ is a cut edge of $H$ if and only if $e$ is a cut vertex of $H^T$.
\item Let $|V| \ge 2$ and let $v \in V$ be such that $\{ v \} \not\in E$. Then $v$ is a cut vertex of $H$ if and only if $v^T$ is a cut edge of $H^T$.
\end{enumerate}
\end{cor}

\begin{proof}
\begin{enumerate}
\item First, since $H$ has no empty edges, by Theorem~\ref{the:cutG}, $e$ is a cut edge of $H$ if and only if it is a cut e-vertex of $\G(H)$, and hence if and only if $e$ is a cut v-vertex of $\G(H^T)$. On the other hand, since $e$ contains no pendant vertices of $H$, we have that $\{ e \} \not\in V^T$. Also, $H^T$ has no empty edges since $H$ has no isolated vertices. Hence by Theorem~\ref{the:cutG}, $e$ is a cut vertex of $H^T$ if and only if $e$ is a cut v-vertex of $\G(H^T)$. The result follows.
\item By Theorem~\ref{the:cutG}, since $H$ has no empty edges and $\{ v \} \not\in E$, vertex $v$ is a cut vertex of $H$ if and only if $v$ is a cut v-vertex of $\G(H)$, and hence if and only if $v^T$ is a cut e-vertex of $\G(H^T)$. Again by Theorem~\ref{the:cutG}, since $H^T$ has no empty edges, this is the case if and only if $v^T$ is a cut edge of $H^T$.
\end{enumerate}
\vspace{-1.6cm}
\end{proof}

\begin{cor}
Let $H=(V,E)$ be a hypergraph with $|V|\ge 2$, $|E| \ge 1$, and with neither empty edges nor isolated vertices. Furthermore,  let $v$ be a cut vertex such that $\{ v \} \not\in E$. Then $\omega(H\b v) \le \omega(H)+\deg_H(v)-1$.
\end{cor}

\begin{proof}
Consider the dual $H^T$ of $H$. Since $v$ is a cut vertex of $H$ and $\{ v \} \not\in E$, by Corollary~\ref{cor:cuts}, the edge $v^T$ of $H^T$ is a cut edge, and hence $\omega(H^T-v^T) \le \omega(H^T)+|v^T|-1$ by Lemma~\ref{lem:cut-edge}. By Corollary~\ref{cor:omega} we have $\omega(H^T)=\omega(H)$, and by Lemma~\ref{lem:delT}, we have $|v^T|=\deg_H(v)$. It remains to show that $\omega(H^T-v^T)=\omega(H\b v)$. Using Corollary~\ref{cor:omega} and Lemma~\ref{lem:delete}, we have $$\omega(H^T-v^T)=\omega(\G(H^T-v^T))=\omega(\G(H^T) \b v^T)=\omega(\G(H) \b v))=\omega(\G(H \b v))=\omega(H\b v)$$
since $H^T-v^T$ has no empty edges, since $\G(H^T-v^T)=\G(H^T) \b v^T$, and since $\G(H^T) \b v^T$ is isomorphic to $\G(H) \b v$, which in turn is equal to $\G(H \b v)$ because $\{ v \} \not\in E$.

We conclude that $\omega(H\b v) \le \omega(H)+\deg_H(v)-1$.
\end{proof}

A graph with a cut edge and at least three vertices necessarily possesses a cut vertex. Here is the analogue for hypergraphs.
\begin{theo}\label{the:cut-e-v}
Let $H=(V,E)$ be a hypergraph with a cut edge $e$ such that for some non-trivial connected component $H'$ of $H-e$, we have $|e \cap V(H')|=1$. Then $H$ has a cut vertex.
\end{theo}

\begin{proof}
We may assume $H$ is connected. Let $H'$ and $H''$ be two connected components of $H-e$, with $H'$ non-trivial and $e \cap V(H')=\{ u \}$. Take any $x \in V(H')-\{ u \}$ and $y \in V(H'')$. Since $e$ is a cut edge, every $(x,y)$-path $P$ in $H$ must contain the edge $e$, and since $u$ is the only vertex of $e$ in $V(H')$, any such path $P$ must also contain $u$ as an anchor vertex. Hence $x$ and $y$ are disconnected in $H\b u$, and  $u$ is a cut vertex of $H$.
\end{proof}

\begin{cor}\label{cor:cut-e-v}
Let $H=(V,E)$ be a connected hypergraph with a strong cut edge $e$ such that $|e| < |V|$. Then $H$ has a cut vertex.
\end{cor}

\begin{proof}
Let $H_1,\ldots,H_k$ be the connected components of $H-e$. By Theorem~\ref{the:strong-cut}, the edge $e$ contains exactly one vertex from each $H_i$ (for $i=1,\dots,k$), and so $k=|e|<|V|$. Hence $|V(H_i)|\ge 2$ for at least one connected component $H_i$, and $|e \cap V(H_i)|=1$ since $e$ is a strong cut edge. It follows by Theorem~\ref{the:cut-e-v} that $H$ has a cut vertex.
\end{proof}


\subsection{Blocks and non-separable hypergraphs}

Throughout this section, we shall assume that our hypergraphs are connected and have no empty edges. We begin by extending the notion of a cut vertex as follows.

\begin{defn}{\rm
Let $H=(V,E)$ be a connected hypergraph without empty edges. A vertex $v \in V$ is a {\em separating vertex} for $H$ if $H$ decomposes into two non-empty connected hypersubgraphs with just vertex $v$ in common. That is, $H=H_1 \oplus H_2$, where  $H_1$ and $H_2$ are two non-empty connected hypersubgraphs of $H$ with  $V(H_1) \cap V(H_2)=\{ v \}$.
}
\end{defn}

\begin{theo}\label{the:sep-cut}
Let $H=(V,E)$ be a connected hypergraph without empty edges, with $|V|\ge 2$ and $v \in V$.
\begin{enumerate}
\item If $v$ is a cut vertex of $H$, then $v$ is a separating vertex of $H$.
\item If $v$ is a separating vertex of $H$ and  $\{ v \} \not\in E$, then $v$ is a cut vertex of $H$.
\end{enumerate}
\end{theo}

\begin{proof}
\begin{enumerate}
\item Assume $v$ is a cut vertex of $H$, let $V_1$ be the vertex set of one connected component of $H\b v$, and let $V_2=V(H \b v)-V_1$. Furthermore, let $H_1$ and $H_2$ be the subhypergraphs induced by the sets $V_1 \cup \{ v \}$ and $V_2 \cup \{ v \}$, respectively, so that $E(H_i)=\mset{ e \cap (V_i \cup \{ v \}): e \in E, e \cap (V_i \cup \{ v \}) \ne \emptyset }$ for $i=1,2$. Clearly $V(H_1) \cap V(H_2)=\{ v \}$. We show that $H_1$ and $H_2$ are in fact hypersubgraphs of $H$ with just vertex $v$ in common.

Take any edge $e \in E$ and suppose $e \cap V_i \ne \emptyset$ for both $i=1,2$. Let $e'=e \cap (V_1 \cup V_2)$. Then $e'$ is an edge of $H \b v$ with vertices in both $V_1$ and $V_2$,  contradicting the fact that $V_1$ is a connected component of $H \b v$. Hence either $e \subseteq V(H_1)$ or $e \subseteq V(H_2)$, and hence either $e \in E(H_1)$ or $e \in E(H_2)$, showing that $H$ decomposes into hypersubgraphs $H_1$ and $H_2$ with just vertex $v$ in common.

To see that each $H_i$ is connected, note that every vertex $x \in V_i$ is connected to $v$ in $H$, and hence also in $H_i$. Since $H_1$ and $H_2$ are non-trivial and connected, they must be non-empty.

Thus $v$ is a separating vertex for $H$.

\item Assume $v$ is a separating vertex of $H$ such that  $\{ v \} \not\in E$. Let $H_1$ and $H_2$ be non-empty connected hypersubgraphs of $H$ with just vertex $v$ in common such that $H=H_1 \oplus H_2$. Hence either $e \in E(H_1)$ or $e \in E(H_2)$ for all $e \in E$. For each $i=1,2$, since hypergraph $H_i$ is non-empty and connected without edges of the form $\{ v \}$, there exists a vertex $v_i \in V(H_i)-\{v \}$  connected to $v$ in $H_i$. We can now see that vertices $v_1$ and $v_2$ are connected in $H$ but not in $H \b v$, since every $(v_1,v_2)$-path in $H$ must contain $v$ as an anchor vertex. It follows that $H \b v$ is disconnected, and so $v$ is a cut vertex of $H$.
\end{enumerate}
\vspace{-1.4cm}
\end{proof}

\noindent Observe that the additional condition in the second statement of the theorem cannot be omitted: a vertex incident with a singleton edge and at least one more edge (which, as we show below, is necessarily a separating vertex) need not be a cut vertex. A simple example is a hypergraph $H=(V,E)$ with $V=\{ u,v \}$ and $E=\{ e_1,e_2 \}$ for $e_1=\{ v \}$ and $e_2=\{ u,v \}$. Then $v$ is a separating vertex of $H$ since $H=H_1 \oplus H_2$ for $H_1=(\{ v \},\{ e_1 \})$ and $H_2=(\{ u,v \}, \{ e_2 \})$, so $v$ is a separating vertex. However, $v$ is not a cut vertex since $H \b v=(\{ u \}, \{ \{ u \} \})$ is connected.

\begin{lemma}\label{lem:sep-vx}
Let $H=(V,E)$ be a connected hypergraph without empty edges, with $|E|\ge 2$, and with $v \in V$ such that $\{ v \} \in E$. Then $v$ is a separating vertex for $H$.
\end{lemma}

\begin{proof}
Since $H$ is connected and has at least two (non-empty) edges, it must have  at least two edges incident with $v$. Let $e_1=\{ v \}$ and $e_2$ be another edge incident with $v$. Furthermore, let $H_1=(\{ v \}, \{ e_1 \})$ and $H_2=(V,E-\{ e_1 \})$. Then $H_1$ and $H_2$ are two non-empty connected hypersubgraphs of $H$ with just vertex $v$ in common such that $H=H_1 \oplus H_2$. Hence $v$ is a separating vertex for $H$.
\end{proof}

Recall that in a graph without loops, separating vertices are precisely the cut vertices. Hence these two terms are equivalent for the incidence graph of a hypergraph. Next, we determine the correspondence between separating vertices of a hypergraph and separating vertices (cut vertices) of its incidence graph.

\begin{theo}\label{the:sepvxG}
Let $H=(V,E)$ be a connected hypergraph without empty edges, and $G=\G(H)$ be its incidence graph. Take any $v \in V$. Then $v$ is a separating vertex of $H$ if and only if it is a separating vertex (cut vertex) of $G$.
\end{theo}

\begin{proof}
If $|V|\ge 2$ and $\{v\} \not\in E$, then by Theorem~\ref{the:sep-cut}, $v$ is a separating vertex of $H$ if and only if it is a cut vertex of $H$ and therefore,  by Theorem~\ref{the:cutG}, if and only if it is a cut vertex (separating vertex) of $G$.

Assume $e=\{ v \} \in E$. If $v$ is a separating vertex of $H$, then it must be incident with another edge $e'$. Hence in the graph $G \b v$, vertex $e$ is an isolated vertex and $e'$ lies in another connected component, showing that $v$ is a cut vertex for $G$. Conversely, if $v$ is a cut vertex of $G$, then $G$ must contain e-vertices adjacent to $v$ other than $e$, and hence $H$ contains edges incident with $v$ other than $e$.  Hence, by Lemma~\ref{lem:sep-vx}, $v$ is a separating vertex of $H$.

The remaining case is that $|V|=1$ and $\{ v \} \not\in E$. Then $H$ must be empty, $G$ is a trivial graph, and $v$ is a separating vertex for neither. \end{proof}

\begin{cor}\label{cor:sep^T}
Let $H=(V,E)$ be a connected non-empty hypergraph with neither empty edges nor isolated vertices, and let $H^T$ be its dual. Let $v \in V$ and $e \in E$, and let $v^T$ and $e$ be the corresponding edge and vertex, respectively, in $H^T$. Then:
\begin{enumerate}
\item $v$ is a separating vertex of $H$ if and only if $v^T$ is a cut edge of $H^T$.
\item $e$ is a cut edge of $H$ if and only if it is a separating vertex of $H^T$.
\end{enumerate}
\end{cor}

\begin{proof} Observe that by Corollary~\ref{cor:omega}, $H^T$ is connected since $H$ is. Clearly, it is also non-empty with neither empty edges nor isolated vertices.
\begin{enumerate}
\item By Theorem~\ref{the:sepvxG}, $v$ is a separating vertex of $H$ if and only if it is a cut vertex of its incidence graph $\G(H)$, and by Theorem~\ref{the:cutG}, $v^T$ is a cut edge of $H^T$ if and only if it is a cut vertex of $\G(H^T)$. Since $\G(H)$ and $\G(H^T)$ are isomorphic with an isomorphism mapping $v$ to $v^T$, the result follows.
\item Interchanging the roles of $H$ and $H^T$, this statement follows from the previous one.
\end{enumerate}
\vspace{-10mm}
\end{proof}

We shall now define blocks of a hypergraph, and in the rest of this section, investigate their properties.

\begin{defn}{\rm
A connected hypergraph without empty edges that has no separating vertices is called {\em non-separable}. A {\em block} of a hypergraph $H$ is a maximal non-separable hypersubgraph of $H$.
}
\end{defn}

\begin{lemma}\label{lem:emptyB}
Let $H$ be a connected hypergraph without empty edges and $B$ an empty block of $H$. Then $H=B$, and $H$ is empty and trivial.
\end{lemma}

\begin{proof}
Since $B$ is empty and connected, it contains a single vertex, say $v$. If $H$ is non-empty, then it contains an edge $e$ incident with $v$. But then $(e,\{ e \})$ is a non-separable hypersubgraph of $H$ that properly contains the block $B$, a contradiction. Hence $H$ is empty. Since it is connected, it must also be trivial (that is, $V=\{ v \}$). Consequently, $H=B$.
\end{proof}

In a graph, every cycle is contained within a block. What follows is the analogous result for hypergraphs.

\begin{lemma}\label{lem:C}
Let $H$ be a hypergraph without empty edges, $C$ a cycle in $H$, and $\H(C)$ and $\H'(C)$ the hypersubgraph and subhypergraph, respectively, of $H$ associated with $C$ (see Definition~\ref{def:HH'}). Then $\H(C)$ and $\H'(C)$ are non-separable.
\end{lemma}

\begin{proof}
As in Definition~\ref{def:HH'}, let $V_a(C)$, $V_f(C)$, and $E(C)$ be the sets of anchor vertices, floater vertices, and edges of the cycle $C$, respectively. Recall that $\H(C)=(V_a(C) \cup V_f(C),E(C))$ and $\H'(C)=(V_a(C),\mset{e \cap V_a(C): e \in E(C)})$.

To see that $\H(C)$ is non-separable, first observe that it is connected.
Let $G_C$ be the incidence graph of $\H(C)$. Then $G_C$ consists of a cycle $C_G$ with v-vertices and e-vertices alternating, and with additional v-vertices (corresponding to floater vertices of $C$) adjacent to some of the e-vertices of the cycle. Suppose $v \in V$ is a separating vertex of $\H(C)$. By Theorem~\ref{the:sepvxG}, $v$ is then a cut v-vertex of $G_C$. Because $G_C$ is bipartite, every connected component of $G_C\b v$ must contain e-vertices. However, $G_C \b v$ contains the cycle $C_G$ if $v$ is a floater, and the path $C_G \b v$ if $v$ is an anchor, both containing all e-vertices of $G_C$. Thus $G_C\b v$ must have a single connected component, and $G_C$ has no cut vertices, a contradiction.  Hence $\H(C)$ is non-separable.

Similarly it can be shown that $\H'(C)$ is non-separable. (Note that the incidence graph of $\H'(C)$ possesses a Hamilton cycle.)
\end{proof}

We are now ready to show that a hypergraph decomposes into its blocks just as a graph does.

\begin{theo}\label{the:blocks}
Let $H=(V,E)$ be a connected hypergraph without empty edges. Then:
\begin{enumerate}
\item The intersection of any two distinct blocks of $H$ contains no edges and at most one vertex.
\item The blocks of $H$ form a decomposition of $H$.
\item The hypersubgraph $\H(C)$ associated with any cycle $C$ of $H$ is contained within a block of $H$.
\end{enumerate}
\end{theo}

\begin{proof}
\begin{enumerate}
\item Suppose $B_1$ and $B_2$ are distinct blocks of $H$ that share more than just a single vertex. First assume that $B_1$ and $B_2$ have at least  two vertices in common, and let $B=B_1 \cup B_2$ (see Definition~\ref{def:U}). We'll show $B$ is a non-separable hypergraph. First, $B$ is connected since $B_1$ and $B_2$ are connected with intersecting vertex sets. Take any $v \in V(B)$. Can $v$ be a separating vertex of $B$? Since $B_1$ and $B_2$ are non-separable, $v$ is not a separating vertex in either block, and hence by Theorem~\ref{the:sep-cut}, $v$ is not a cut vertex in either block, and $B_1 \b v$ and $B_2 \b v$ are connected. Since $B \b v=(B_1 \b v) \cup (B_2 \b v)$, and $B_1 \b v$ and $B_2 \b v$ are connected with at least one common vertex, it follows that $B \b v$ is connected. Hence $v$ is not a cut vertex of $B$. If $v$ is a separating vertex of $B$, then by Theorem~\ref{the:sep-cut}, we must have $e \in E(B)$ for $e=\{ v \}$. Hence, without loss of generality, $e \in E(B_1)$. But then, by Lemma~\ref{lem:sep-vx}, $v$ is a separating vertex of $B_1$, because $B_1$ is connected with at least two vertices and hence at least one more edge incident with $v$  --- a contradiction. Hence $B$ is a non-separable hypersubgraph of $H$, and since $B_1$ and $B_2$ are maximal non-separable hypersubgraphs of $H$, we must have $B_1=B_2=B$, a contradiction.

    Hence $B_1$ and $B_2$ have at most one common vertex. Suppose they have a common edge $e$. Then $e$ must be a singleton edge, say $e = \{ v \}$. If $B_1$ or $B_2$ contains another edge, then by Lemma~\ref{lem:sep-vx}, $v$ is a separating vertex for this block, a contradiction. Hence $B_1=B_2=(\{ v \},\{ e \})$, again a contradiction. We conclude that $B_1$ and $B_2$ have no common edges and at most one common vertex.

\item If $H$ has an isolated vertex $v$, then $V=\{ v\}$ and $E = \emptyset$, so $H$ is a block.  Hence assume every vertex of $H$ is incident with an edge. Observe that any $e \in E$ induces a hypersubgraph $(e,\{ e \})$ of $H$, which is non-separable and hence is a hypersubgraph of a block of $H$. Thus every edge and every vertex of $H$ is contained in a block. Since by the first statement of the theorem no two blocks share an edge, every edge of $H$ is contained in exactly one block, and $H$ is an edge-disjoint union of its blocks.

\item By Lemma~\ref{lem:C}, the hypersubgraph $\H(C)$ of a cycle $C$ is non-separable, and hence a hypersubgraph of a block of $H$.
\end{enumerate}
\vspace{-15mm}
\end{proof}

The next lemma will be used several times.

\begin{lemma}\label{lem:Bsubgr}
Let  $H'$ be a connected hypersubgraph of a connected hypergraph $H$ without empty edges, and $v \in V(H')$. If $H'$ contains edges of two blocks of $H$ that intersect in vertex $v$, then $v$ is a separating vertex of $H'$.
\end{lemma}

\begin{proof}
Let $B_1$ and $B_2$ be distinct blocks of $H$ intersecting in vertex $v$ such that $H'$ contains an edge from each of them. Note that $B_1$ and $B_2$ must both be non-empty, since otherwise $B_1=B_2=H$ is empty by Lemma~\ref{lem:emptyB}. If $B_1$ is trivial, then $\{ v \} \in E(B_1) \cap E(H')$, and $v$ is a separating vertex of $H'$ by Lemma~\ref{lem:sep-vx}. Hence assume $B_1$ and $B_2$ are both non-trivial. Since $H'$ is connected, we may assume there exist a vertex $x$ adjacent to $v$ in $B_1 \cap H'$ via edge $e_1$, and a vertex $y$ adjacent to $v$ in $B_2 \cap H'$ via edge $e_2$. Suppose there exists an $(x,y)$-path $P$ in $H' \b v$. Then $Pye_2ve_1x$ is a cycle in $H'$ containing vertices $v$, $x$, and $y$. By Statement (3) of Theorem~\ref{the:blocks}, these three vertices lie in a common block $B$, and by Statement (1) of the same result, $B_1=B=B_2$, a contradiction. Hence $x$ and $y$ must lie in distinct connected components of $H' \b v$. It follows that $v$ is a cut vertex of $H'$, and hence a separating vertex of $H'$ by Theorem~\ref{the:sep-cut}.
\end{proof}

\begin{theo}\label{the:blocks-sepvx}
Let $H=(V,E)$ be a connected hypergraph without empty edges, and $v \in V$. Then $v$  is a separating vertex of $H$ if and only if it lies in more than one block.
\end{theo}

\begin{proof}
Assume $v$ is a separating vertex of $H$. Then $H=H_1 \oplus H_2$, where $H_1$ and $H_2$ are non-empty connected hypersubgraphs with just vertex $v$ in common. Hence there exist $e_1 \in E(H_1)$ and $e_2 \in E(H_2)$ such that $v \in e_1 \cap e_2$. By Statement (2) of Theorem~\ref{the:blocks}, there exist blocks $B_1$ and $B_2$ of $H$ such that $e_1 \in E(B_1)$ and $e_2 \in E(B_2)$.

Observe that $B_1 \cap H_1$ is connected: since $B_1$ is connected, and $H_1$ and $H_2$ intersect only in the vertex $v$, every vertex in $B_1 \cap H_1$ is connected to $v$ in $B_1 \cap H_1$. Similarly, $B_1 \cap H_2$ is connected.

Suppose that $B_1=B_2$. Then $B_1 =(B_1 \cap H_1) \oplus (B_1 \cap H_2)$ with $B_1 \cap H_1$ and $B_1 \cap H_2$ connected, non-empty, and intersecting only in vertex $v$ --- a contradiction, because $B_1$ is non-separable. Hence $B_1$ and $B_2$ must be distinct blocks of $H$ containing vertex $v$.

Conversely, assume that $v$ lies in the intersection of distinct blocks $B_1$ and $B_2$ of $H$. By Lemma~\ref{lem:emptyB}, $B_1$ and $B_2$ are non-empty. Then $H$ itself is a connected hypersubgraph of $H$ containing edges from two blocks of $H$ that intersect in $v$. It follows from Lemma~\ref{lem:Bsubgr} that $v$ is a separating vertex of $H$.
\end{proof}

Theorems~\ref{the:blocks} and \ref{the:blocks-sepvx} show that a block graph of a hypergraph can be defined just as for graphs. Namely, let $H$ be a connected hypergraph without empty edges, $S$ the set of its separating vertices, and ${\cal B}$ the collection of its blocks. Then the {\em block graph} of $H$ is the bipartite graph with vertex bipartition $\{S,{\cal B}\}$ and edge set $\{ vB: v \in S, B \in {\cal B},v \in V(B) \}$. From the third statement of Theorem~\ref{the:blocks} it then follows that the block graph of $H$ is a tree.

Next, we show that blocks of a hypergraph correspond to maximal clusters of blocks of its incidence graph, to be defined below.

\begin{defn}\label{def:cluster}{\rm
Let $H=(V,E)$ be a connected hypergraph without empty edges, and $G=\G(H)$ its incidence graph. A {\em cluster of blocks} of $G$ is a connected union of blocks of $G$, no two of which share a v-vertex.}
\end{defn}

\begin{theo}\label{the:HGblocks}
Let $H=(V,E)$ be a connected hypergraph without empty edges and $H'$ its hypersubgraph, and let $G=\G(H)$ and $G'=\G(H')$ be their incidence graphs, respectively. Then $H'$ is a block of $H$ if and only if $G'$ is a maximal cluster of blocks of $G$.
\end{theo}

\begin{proof}
Assume $H'$ is a block of $H$. We first show that $G'=\G(H')$ is a cluster of blocks of $G$. Let $C$ be the union of all blocks of $G$ that have a common edge with $G'$. Observe that since $H'$ is connected and has no empty edges, $G'$ is connected by Theorem~\ref{the:conn}, and consequently $C$ is connected. Suppose that two distinct blocks of $C$, say $B_1$ and $B_2$, share a v-vertex of $G$. Since $G'$ contains an edge from both $B_1$ and $B_2$, $v$ is a separating vertex of $G'$ by Lemma~\ref{lem:Bsubgr}. However, by Theorem~\ref{the:sepvxG}, $v$ is then a separating vertex of the block $H'$ of $H$, a contradiction.

Hence no two distinct blocks in $C$ intersect in a v-vertex, and $C$ is a cluster of blocks of $G$. Let $C^*$ be a maximal cluster of blocks of $G$ containing $C$. Then $C^*$ is connected, and has no separating v-vertices by Theorem~\ref{the:blocks-sepvx}. Since $C^*$ is maximal, no e-vertex of $C^*$ can be contained in a block not in $C^*$. Consequently, for every e-vertex $e$ of $C^*$, all edges of the form $ev$ (for $v \in V$) are contained in $C^*$. Hence, by Lemma~\ref{lem:sub-inc}, $C^*$ is the incidence graph of a hypersubgraph $H^*$ of $H$. Now $H^*$ is connected and has no separating vertices since $C^*$ is connected and has no separating v-vertices. Moreover, $H^*$ contains the block $H'$. We conclude that $H^*=H'$ and $C^*=G'$. It follows that $G'$ is a maximal cluster of blocks of $G$.

Conversely, let $G'$ be a maximal cluster of blocks of $G$. Then for every e-vertex $e$ of $G'$, all edges of $G$ of the form $ev$ (for $v \in V$ such that $v \in e$) must be in $G'$, so by Lemma~\ref{lem:sub-inc}, $G'=\G(H')$ for some hypersubgraph $H'$ of $H$. Since $G'$ is connected and has no separating v-vertices, $H'$ is connected and non-separable. Hence $H'$ is  contained in a block $B$ of $H$. By the previous paragraph, $\G(B)$ is a maximal cluster of blocks of $G$, and it also contains the maximal cluster $G'$. We conclude that $\G(B)=G'$, that is,  $G'$ is the incidence graph of a block of $H$.
\end{proof}

The next corollary is immediate.

\begin{cor}
Let $H=(V,E)$ be a connected hypergraph without empty edges, and $G=\G(H)$ its incidence graph. Then $H$ is non-separable if and only if $G$ is a cluster of blocks of $G$.
\end{cor}

To complete the discussion on the blocks of the incidence graph of a hypergraph, we show the following.

\begin{theo}
Let $H=(V,E)$ be a non-separable hypergraph with at least two edges of cardinality greater than 1. Let $G=\G(H)$ be its incidence graph and $x$ a cut vertex of $G$. Then $x \in E$ and $x$ is a weak cut edge of $H$.
\end{theo}

\begin{proof}
If $x \in V$, then $x$ is a separating vertex of $H$ by Theorem~\ref{the:sepvxG}, a contradiction. Hence
 $x \in E$, and  $x$ is a cut edge of $H$ by Theorem~\ref{the:cutG}. Suppose $x$ is a strong cut edge. If  $|x|<|V|$, then $H$ has a cut vertex by Corollary~\ref{cor:cut-e-v}, amd hence a separating vertex by Theorem~\ref{the:sep-cut}, a contradiction. Hence $|x|=|V|$, and by Theorem~\ref{the:strong-cut}, $H-x$ has exactly $|x|$ connected components, implying that $x$ is the only edge of $H$ of cardinality greater than 1, a contradiction. Hence $x$ must be a weak cut edge of $H$.
\end{proof}

In the last four theorems we attempt to generalize the following classic result from graph theory.

\begin{theo}\cite{BonMur}\label{the:BonMur} \begin{enumerate}
\item A connected graph is non-separable if and only if any two of its edges lie on a common cycle.
\item A connected graph with at least three vertices has no cut vertex if and only if any two of its vertices lie on a common cycle.
\end{enumerate}
\end{theo}

\begin{theo}
Let $H=(V,E)$ be a non-separable hypergraph with $|V|\ge 2$ and $|E|\ge 2$, and let $G=\G(H)$ be its incidence graph. Assume in addition that $V \not\in E$ and that $H$ has no weak cut edges. Then any two distinct vertices of $H$ and any two distinct edges of $H$ lie on a common cycle.
\end{theo}

\begin{proof}
Suppose that $G$ has a separating vertex $x$. If $x \in V$, then by Theorem~\ref{the:sepvxG}, $x$ is a separating vertex of $H$, a contradiction. Thus $x \in E$, and $x$ is a cut edge of $H$ by Theorem~\ref{the:cutG}. By assumption, $x$ is a strong cut edge and $|x|<|V|$. Hence $H$ has a cut vertex, and hence a separating vertex, by Corollary~\ref{cor:cut-e-v} and Theorem~\ref{the:sep-cut}, respectively --- a contradiction.

Hence $G$ has no cut vertex, and by Theorem~\ref{the:BonMur}, any two vertices of $G$ lie on a common cycle. It then follows from Lemma~\ref{lem:W-W_G} that any two vertices, and any two edges, of $H$ lie on a common cycle.
\end{proof}

\begin{theo}
Let $H=(V,E)$ be a connected hypergraph with $|V|\ge 2$, without edges of cardinality less than 2, and without vertices of degree less than 2. Then the following are equivalent:
\begin{enumerate}
\item $H$ has no separating vertices and no cut edges.
\item Every pair of elements from $V \cup E$ lie on a common cycle.
\item Every pair of vertices lie on a common cycle.
\item Every pair of edges lie on a common cycle.
\end{enumerate}
\end{theo}

\begin{proof}
Let $G=\G(H)$ be the incidence graph of $H$.

$(1) \Rightarrow (2)$: Since $H$ has no separating vertices and no cut edges, $G$ has no cut vertices by Theorems~\ref{the:sepvxG} and \ref{the:cutG}. Hence by Theorem~\ref{the:BonMur}, since $|V(G)|\ge 3$, every pair of vertices of $G$ lie on a common cycle in $G$, and therefore every pair of elements from $V \cup E$ lie on a common cycle in $H$.

$(2) \Rightarrow (3)$: This is obvious.

$(3) \Rightarrow (4)$: Since every pair of vertices of $H$ lie on a common cycle in $H$, every pair of v-vertices of $G$ lie on a common cycle in $G$. Consequently, by Theorem~\ref{the:blocks}, all v-vertices of $G$ are contained in the same block $B$, and if $G$ has any other blocks, then they are isomorphic to $K_2$. Let $B_1$ be one of these ``trivial'' blocks, and let $e$ be its e-vertex. Then $\deg_G(e)=1$ ---  a contradiction, since $H$ has no singleton edges. It follows that $G$ has no ``trivial'' blocks, and hence no cut vertices. Therefore every pair of e-vertices of $G$ lie on a common cycle in $G$, and every pair of edges of $H$ lie on a common cycle in $H$.

$(4) \Rightarrow (1)$: Since every pair of edges of $H$ lie on a common cycle in $H$, every pair of e-vertices of $G$ lie on a common cycle in $G$. Consequently, all e-vertices of $G$ are contained in the same block $B$, and if $G$ has any other blocks, then they are isomorphic to $K_2$. Let $B_1$ be one of these ``trivial'' blocks, and let $v$ be its v-vertex. Then $\deg_G(v)=1$ ---  a contradiction, since $H$ has no pendant vertices. It follows that $G$ has no ``trivial'' blocks, and hence no cut vertices. Therefore $H$ has no separating vertices and no cut edges by Theorems~\ref{the:sepvxG} and \ref{the:cutG}, respectively.
\end{proof}

\begin{theo}\label{the:no-cut-e}
Let $H=(V,E)$ be a connected hypergraph with $|V|\ge 2$, without edges of cardinality less than 2, and without vertices of degree less than 2. Then the following are equivalent:
\begin{enumerate}
\item $H$ has no cut edges.
\item Every pair of elements from $V \cup E$ lie on a common strict closed trail.
\item Every pair of vertices lie on a common strict closed trail.
\item Every pair of edges lie on a common strict closed trail.
\end{enumerate}
\end{theo}

\begin{proof}
Let $G=\G(H)$ be the incidence graph of $H$.

$(1) \Rightarrow (2)$: Since $H$ has no cut edges, $G$ has no cut e-vertices by Theorem~\ref{the:cutG}. Take any two elements $x_0$ and $x_k$ of $V \cup E$. We construct a strict closed trail in $H$ containing $x_0$ and $x_k$ as follows. Let $B_1$ and $B_k$ be blocks of $G$ containing $x_0$ and $x_k$, respectively, and let $P=B_1 x_1B_2 \ldots B_{k-1}x_{k-1}B_k$ be the unique $(B_1,B_k)$-path in the block tree of $G$. Here, of course, $B_1,\ldots,B_k$ are blocks of $G$, $x_1,\ldots,x_{k-1}$ are separating (cut) vertices of $G$, and each separating vertex $x_i$ (necessarily a v-vertex) is shared between blocks $B_i$ and $B_{i+1}$. (We may assume that vertex $x_0$ does not lie in block $B_2$, and $x_k$ does not lie in $B_{k-1}$, otherwise the path $P$ may be shortened accordingly.) By Theorem~\ref{the:BonMur}, each pair of vertices $x_{i-1}$ and $x_i$, for $i=1,\ldots,k$,  lie on a common cycle $C_i$ within block $B_i$. Note that these cycles $C_1,\ldots,C_k$ are pairwise edge-disjoint and intersect only in the v-vertices $x_1,\ldots,x_{k-1}$. Let $T=C_1 \oplus \ldots \oplus C_k$. Then $T$ is a closed trail in $G$ containing $x_0$ and $x_k$ that does not repeat any e-vertices. (We count the first and last vertex of a closed trail --- which are identical --- as one occurrence of this vertex.) We conclude that every pair of vertices of $G$ lie on a common closed trail in $G$ that traverses each e-vertex at most once. Therefore, by Lemma~\ref{lem:W-W_G}, every pair of elements from $V \cup E$ lie on a common strict closed trail in $H$.

$(2) \Rightarrow (3)$: This is obvious.

$(3) \Rightarrow (4)$: Since every pair of vertices of $H$ lie on a common strict closed trail in $H$, every pair of v-vertices of $G$ lie on a common closed trail in $G$ that visits each e-vertex at most once. Suppose $G$ has a cut e-vertex $e$. Let let $v_1$ and $v_2$ be two v-vertices in distinct connected components of $G \b e$. Since $e$ is a cut vertex, $v_1$ and $v_2$ are disconnected in $G \b e$. On the other hand, by assumption,  $v_1$ and $v_2$ lie on a closed trail $T$ that traverses $e$ at most once. Hence $T \b e$ contains a $(v_1,v_2)$-path of $G \b e$, a contradiction.
Consequently, $G$ has no cut e-vertices, which implies (as seen in the previous paragraph) that any two vertices --- and hence any two e-vertices ---  lie on a common closed trail in $G$ that does not repeat any e-vertices. Therefore every pair of edges of $H$ lie on a common strict closed trail in $H$.

$(4) \Rightarrow (1)$: Since every pair of edges of $H$ lie on a common strict closed trail in $H$, every pair of e-vertices of $G$ lie on a common closed trail in $G$ that does not repeat any e-vertices. Suppose $G$ has a cut e-vertex $e$. Since $H$ has no vertices of degree less than 2, $G \b e$ has no trivial connected components; that is, each connected component of $G \b e$ contains $e$-vertices. Let $e_1$ and $e_2$ be two e-vertices from distinct connected components of $G \b e$. Then $e_1$ and $e_2$ are disconnected in $G \b e$. On the other hand, by assumption,  $e_1$ and $e_2$ lie on a closed trail $T$ that traverses $e$ at most once. Hence $T \b e$ contains an $(e_1,e_2)$-path of $G \b e$, a contradiction. It follows that $G$ has no cut e-vertices, and $H$ has no cut edges by Theorem~\ref{the:cutG}.
\end{proof}

We conclude with the dual version of the previous theorem.

\begin{cor}
Let $H=(V,E)$ be a connected hypergraph with $|E|\ge 2$, without edges of cardinality less than 2, and without vertices of degree less than 2. Then the following are equivalent:
\begin{enumerate}
\item $H$ has no separating vertices.
\item Every pair of elements from $V \cup E$ lie on a common pseudo cycle.
\item Every pair of edges lie on a common pseudo cycle.
\item Every pair of vertices lie on a common pseudo cycle.
\end{enumerate}
\end{cor}

\begin{proof}
Let $H^T$ be the dual of $H$, and observe that (by Corollary~\ref{cor:omega} and since $H$ must have at least 2 vertices) $H^T$ satisfies the assumptions of Theorem~\ref{the:no-cut-e}. Since separating vertices of $H$ correspond precisely to cut edges of $H^T$ by Corollary~\ref{cor:sep^T}, and pseudo cycles of $H$ to strict closed trails of $H^T$ by Lemma~\ref{lem:W^T}, the corollary follows easily from Theorem~\ref{the:no-cut-e}.
\end{proof}

\section{Conclusion}

In this paper, we generalized several concepts related to connection in graphs to hypergraphs. While some of these concepts generalize naturally in a unique way, or behave in hypergraphs similarly to graphs, other concepts lend themselves to more than one natural generalization, or reveal surprising new properties. Many more concepts from graph theory remain unexplored for hypergraphs, and we hope that our work will stimulate more research in this area.


\begin{thebibliography}{99}
\bibitem{Ber} Claude Berge, {\em Graphs and Hypergraphs}, North-Holland, New York, 1976.
\bibitem{Ber1} Claude Berge, {\em Hypergraphs, Combinatorics of finite sets}, North-Holland Mathematical Library 45, North-Holland Publishing, Amsterdam, 1989.
\bibitem{BonMur0} J. A. Bondy, U. S. R. Murty, {\em Graph theory with applications}, American Elsevier Publishing, New York, 1976.
\bibitem{BonMur} J. A. Bondy, U. S. R. Murty, {\em Graph theory}. Graduate Texts in Mathematics {\bf 244}, Springer, New York, 2008.
\bibitem{Bre} Alain Bretto, {\em Hypergraph Theory, An Introduction}, Springer, 2013.
\bibitem{Duc} Pierre Duchet, Hypergraphs, in {\em Handbook of combinatorics}, edited by R. L. Graham, M. Gr\"{o}tschel, and L. Lov\'{a}sz,  381--432, Elsevier, Amsterdam, 1995.
\bibitem{Vol0} Vitaly I. Voloshin, {\em Coloring mixed hypergraphs: theory, algorithms and applications}, Fields Institute Monographs {\bf 17}, American Mathematical Society, Providence, RI, 2002.
\bibitem{Vol} Vitaly I. Voloshin, {\em Introduction to graph and hypergraph theory}, Nova Science Publishers, New York, 2009.
\end{thebibliography}
\end{document}